\documentclass{article}

\usepackage{lmodern}
\usepackage[T1]{fontenc}
\usepackage[latin9]{inputenc}  
\usepackage[
    backend=bibtex,
    style=alphabetic,
  ]{biblatex}

\usepackage{a4wide}        
\usepackage{amsmath, amssymb}    
\usepackage{amsthm}                
\usepackage{appendix}
\usepackage{mathtools}
\usepackage{tikz}                   
\usepackage{tikz-cd}                
\tikzcdset{ 
arrow style=tikz
}
\usepackage{graphicx}               
\usepackage{stackrel}          
\usepackage{color}                 
\usepackage[normalem]{ulem}        
\usepackage{enumitem}               
\setlist[itemize]{leftmargin=*}
\setlist[enumerate]{leftmargin=*,label=\roman*),ref=\roman*)}
\newlist{subenumerate}{enumerate}{2}
\setlist[subenumerate]{leftmargin=*,label=\alph*),ref=\alph*)}
\usepackage{blkarray}    

\definecolor{darkblue}{rgb}{0,0,0.6} 
\usepackage[ocgcolorlinks,colorlinks=true, citecolor=darkblue, filecolor=darkblue, linkcolor=darkblue, urlcolor=darkblue]{hyperref}         

\usepackage{thmtools}
\usepackage{thm-restate}

\usepackage{bbm}                  
\usepackage{stix}                  
\usepackage{eucal}       
 
\usepackage{enumitem}
\setlist[enumerate,1]{label={(\arabic*)}}

\DeclareMathOperator{\cofib}{cofib}
\DeclareMathOperator{\fib}{fib}

\DeclareMathOperator{\Eq}{Eq}

\newcommand{\rmB}{\mathrm{B}}

\newcommand{\bbZ}{\mathbb{Z}}

\newcommand{\bbQ}{\mathbb{Q}}

\newcommand{\bbN}{\mathbb{N}}

\newcommand{\bbS}{\mathbb{S}}

\newcommand{\calP}{\mathcal{P}}
\newcommand{\J}{\mathcal{J}}

\newcommand{\calA}{\mathcal A}
\newcommand{\calB}{\mathcal B}
\newcommand{\calC}{\mathcal C}
\newcommand{\calD}{\mathcal D}
\newcommand{\calE}{\mathcal E}

\theoremstyle{plain}
\newtheorem{thm}{Theorem}[section]

\newtheorem{thmx}{Theorem}
\newtheorem{innercustomthm}{Theorem}
\newenvironment{mythmx}[1]
  {\renewcommand\theinnercustomthm{#1}\innercustomthm}
  {\endinnercustomthm}

\newtheorem{lem}[thm]{Lemma}
\newtheorem{cor}[thm]{Corollary}

\newtheorem{prop}[thm]{Proposition}

\theoremstyle{definition}
\newtheorem{defn}[thm]{Definition}

\newtheorem{rem}[thm]{Remark}

\newtheorem{warn}{Warning}

\newtheorem{exam}[thm]{Example}

\theoremstyle{remark}

\theoremstyle{plain}
\newtheorem*{thm*}{Theorem}
\newtheorem*{lem*}{Lemma}
\newtheorem*{cor*}{Corollary}
\newtheorem*{prop*}{Proposition}
\newtheorem*{fact*}{Fact}
\newtheorem*{claim*}{Claim}
\newtheorem*{conj*}{Conjecture}

\theoremstyle{definition}
\newtheorem*{defn*}{Definition}
\newtheorem*{notn*}{Notation}
\newtheorem*{convs*}{Conventions}
\newtheorem*{ackn*}{Acknowledgements}
\newtheorem*{rem*}{Remark}
\newtheorem*{warn*}{Warning}
\newtheorem*{disc*}{Discussion}
\newtheorem*{quest*}{Question}
\newtheorem*{exams*}{Examples}
\newtheorem*{exam*}{Example}
\newtheorem*{constr*}{Construction}
\newtheorem*{goal*}{Goal}
\theoremstyle{remark}

\newcommand{\from}{\colon}
\newcommand{\id}{\mathrm{Id}}

\newcommand{\defeq}{\vcentcolon=}

\newcommand{\op}{{\mathrm{op}}}

\DeclareFontFamily{T1}{cbgreek}{}
\DeclareFontShape{T1}{cbgreek}{m}{n}{<-6>  grmn0500 <6-7> grmn0600 <7-8> grmn0700 <8-9> grmn0800 <9-10> grmn0900 <10-12> grmn1000 <12-17> grmn1200 <17-> grmn1728}{}
\DeclareSymbolFont{quadratics}{T1}{cbgreek}{m}{n}
\DeclareMathSymbol{\qoppa}{\mathord}{quadratics}{19}
\DeclareMathSymbol{\Qoppa}{\mathord}{quadratics}{21}

\newcommand{\Pn}{\mathrm{Pn}}
\DeclareMathOperator{\He}{He}
\DeclareMathOperator{\Fm}{Fm}
\DeclareMathOperator{\Qdot}{Q_\bullet}

\newcommand{\Ctwo}{{\mathrm{C}_2}}
\newcommand{\hCtwo}{{\mathrm{hC}_2}}
\newcommand{\tCtwo}{{\mathrm{tC}_2}}

\newcommand{\Lnpf}{L_n^{p,f}}
\newcommand{\Lnf}{L_n^f}
\newcommand{\LTn}{L_{T(n)}}
\newcommand{\LTi}{L_{T(i)}}

\newcommand{\Einf}{\mathbf{E}_\infty}
\newcommand{\Eone}{\mathbf{E}_1}

\DeclareMathOperator{\flatAdd}{U_\flat}

\newcommand{\CMon}{\mathrm{CMon}}
\newcommand{\CGrp}{\mathrm{CGrp}}

\newcommand{\gp}{\mathrm{gp}}

\newcommand{\CAlg}{\mathrm{CAlg}}

\newcommand{\Nm}{\mathrm{Nm}}

\DeclareMathOperator{\Fun}{Fun}
\DeclareMathOperator{\Map}{Map}
\DeclareMathOperator{\map}{map}
\DeclareMathOperator{\TwAr}{TwAr}
\DeclareMathOperator{\Mod}{Mod}
\DeclareMathOperator{\Perf}{Perf}
\DeclareMathOperator{\Proj}{Proj}

\DeclareMathOperator*{\colim}{colim}
\DeclareMathOperator{\Stab}{Stab}
\DeclareMathOperator{\Hyp}{Hyp}

\newcommand{\A}{\mathcal{A}}
\newcommand{\B}{\mathcal{B}}
\newcommand{\C}{\mathcal{C}}

\newcommand{\An}{\mathcal{S}}
\newcommand{\Spaces}{\mathcal{S}}

\newcommand{\PP}{\mathcal{P}}

\newcommand{\Span}{\mathrm{Span}}

\newcommand{\Sp}{\mathrm{Sp}}
\newcommand{\Spcn}{\mathrm{Sp}^\mathrm{cn}}

\newcommand{\Cat}{\mathrm{Cat}_\infty}

\newcommand{\Catex}{\mathrm{Cat}^{\mathrm{ex}}_\infty}
\newcommand{\Cath}{\mathrm{Cat}^{\mathrm{h}}_\infty}
\newcommand{\Catp}{\mathrm{Cat}^{\mathrm{p}}_\infty}
\newcommand{\Catpidem}{\mathrm{Cat}^{\mathrm{p}}_\mathrm{idem}}
\newcommand{\Catpd}{\mathrm{Cat}^{\mathrm{pd}}}
\newcommand{\Catadd}{\mathrm{Cat}^{\mathrm{add}}_\infty}
\newcommand{\Cataddflat}{\mathrm{Cat}^{\mathrm{add}}_\flat}
\newcommand{\Catsadd}{\mathrm{Cat}^{\mathrm{sadd}}_\infty}
\newcommand{\Catperf}{\mathrm{Cat}^{\mathrm{perf}}_\infty}

\newcommand{\Catah}{\mathrm{Cat}^{\mathrm{ah}}_\flat}
\newcommand{\Catap}{\mathrm{Cat}^{\mathrm{ap}}_\flat}

\newcommand{\Catcoprod}{\mathrm{Cat}^{\amalg}_\infty}

\newcommand{\rmD}{\mathrm{D}}
\newcommand{\DQoppa}{\mathrm{D}_\Qoppa}

\DeclareMathOperator{\Funaq}{Fun^{aq}}
\DeclareMathOperator{\Funq}{Fun^{q}}
\DeclareMathOperator{\GWspace}{\mathcal{GW}}
\DeclareMathOperator{\Lspace}{\mathcal{L}}
\DeclareMathOperator{\Kspace}{\mathcal{K}}
\DeclareMathOperator{\THR}{TH\mathbb R}
\DeclareMathOperator{\K}{{K}}
\DeclareMathOperator{\GW}{{GW}}
\DeclareMathOperator{\KR}{{K}\mathbb{R}}

\DeclareMathOperator{\LL}{{L}}
\DeclareMathOperator{\height}{ht}
\DeclareMathOperator{\End}{End}

\newcommand{\KaroubiK}{\mathbb{K}}
\newcommand{\KaroubiL}{\mathbb{L}}

\newcommand{\loopsinf}{\Omega^\infty}

\definecolor{olivegreen}{rgb}{0,0.6,0}
\definecolor{DefColor}{rgb}{0.6,0.15,0.25}
\newcommand{\mdef}[1]{\textcolor{DefColor}{#1}}

\title{Chromatic Purity in Hermitian K-Theory at $p=2$}
\author{Jordan Levin}
\date{March 2024}

\begin{document}
\maketitle

\begin{abstract}
   In this article we investigate the question of chromatic purity of $\LL$-theory. To do so, we utilize the theory of additive $\GW$ and $\LL$-theory in the language of Poincar\'e categories as laid out in the series of papers \cite{CDHI, CDHII, CDHIII, CDHIV} along with the companion paper \cite{HS21}. We apply this theory to chromatically localised $\LL$-theory at the prime $p=2$ and recover the $\LL$-theoretic analogues of chromatic purity for $\Eone$-rings with involution. From this, we deduce that $\LL$-theory does not exhibit chromatic redshift. We deduce the higher chromatic vanishing of quadratic $\LL$-theory of arbitrary idempotent complete categories, thereby allowing the use of Hermitian trace methods to probe chromatic behaviour of $\GW$ and $\LL$-theory. Finally, we show that for $T(n+1)$-acyclic rings with involution, $T(n+1)$-local $\GW$-theory depends only on $T(n+1)$-local $\K$-theory and the associated duality, thereby proving a chromatic analogue of the homotopy limit problem for $\GW$-theory.
\end{abstract}

\tableofcontents

\section{Introduction}
The $\infty$-category of spectra $\Sp$ is the primary object of study in stable homotopy theory. For each prime $p$, the chromatic perspective on spectra seeks to decompose $\Sp_{(p)}$ - those spectra with $\bbZ_{(p)}$-local homotopy groups - into a sequence of subcategories or ``layers'' each of which captures further refined homotopy-theoretic data. Informally, we can speak of a $p$-local spectrum as being of height $n$ if it lives in the $n$th layer of this sequence. For the sake of the reader we shall present the precise definition of chromatic height. For this, fix a prime $p$ and recall that the $i$th telescope at $p$, $T(i)$, is the spectrum $\bbS/(p^{e_0},\ldots, v_{m-1}^{e_{i-1}}) [v_i^{-1}]$. The precise exponents are irrelevant for our treatment and can safely be ignored. By convention $T(0) = H\bbQ$. 
\begin{defn}\label{defn: chromatic_height}
    Let $X$ be a spectrum. We say that $X$ is of \mdef{height $ \leq n$} if $X$ is $T(i)$-acylic for all $i > n$. We say $X$ is of \mdef{height $\geq n $} if $A$ is not $T(n)$-acyclic. We say that $A$ is of \mdef{height exactly $n$} if it both of height $\leq n$ and $\geq n$. 
\end{defn} 
For $A\in \CAlg$, being of height $\leq n$ is equivalent to being $T(n+1)$-acyclic by \cite{Hah22}. Thus the height of $A$ is the smallest $n$ such that $R \otimes T(n) \not \simeq 0$. Because we will often work with $\Eone$-rings, we are required to use Definition \ref{defn: chromatic_height}. 

Ausoni-Rognes \cite{AR02} computed the homotopy of $\bbS/(3,v_1) \otimes \K(\mathrm{ku}_p)$ and demonstrated that $\K(\mathrm{ku}_p)$ is of height 2 whereas $\mathrm{ku}_p$ is of height 1. They conjectured that this phenomenon should be completely general, algebraic $\K$-theory sends height $n$ spectra to height $n+1$ spectra. One weak formulation of this is the following conjecture.  
\begin{conj*}[Chromatic Redshift]\label{conj: chromatic_redshift}
Suppose $A$ is an $\Einf$-ring  of height exactly $n$. Then $\K(A)$ is an $\Einf$-ring of height height exactly $n+1$.
\end{conj*}
 The proof of this theorem was finally completed in \cite{BSY22} but the machinery required spans a number of articles. Let us draw particular attention to the works \cite{CMNN20a},\cite{LMMT20}, \cite{Yua21}. 
 
 The redshift conjecture can be further refined to a question about \textit{localising invariants} more generally. Roughly speaking, a localising invariant is a functor $E \from \Catperf \to \Sp$ that sends bifibre sequences in $\Catperf$ to exact sequences of spectra. Examples of localising invariants other than $\K$-theory include $\mathrm{THH}$ and $\mathrm{TC}$. This leads to a natural question.
\begin{quest*}[$E$-heightshift]
    Let $E$ be a localising invariant and suppose $\calC \in \Catperf$ has mapping spectra of height $\leq n$, then what about be said about the height of $E(\calC)$?
\end{quest*}
There is an alternative way one might want to generalise this question. Instead, we could start not with stable $\infty$-categories and localising invariants but rather \textit{Poincar\'e} $\infty$-categories and \textit{Poincar\'e} localising invariants. Whereas stable $\infty$-categories capture the behaviour of $\calD^b(A)$ for an $\Eone$-ring $A$, Poincar\'e $\infty$-categories capture the behaviour of $\calD^b(A)$ equipped with a duality coming from an anti-involution $\sigma \from A \to A^\op$. We recall the relevant definitions from \cite{CDHI}.

 \subsection{Recollections}

\begin{defn}\label{defn: poincare_category}
 A \mdef{Hermitian category} is a pair $(\calC,\Qoppa)$ with $\calC$ a stable $\infty$-category and $\Qoppa \from \calC^\op \to \Sp$ a reduced 2-excisive or \mdef{quadratic} functor. To a Hermitian category $(\calC,\Qoppa)$ we can associate a symmetric bilinear functor $\rmB_\Qoppa \from \calC^\op \times \calC^\op \to \Sp$ called the \mdef{cross effect} or \mdef{bilinear part} \footnote{There is also a linear functor $\Lambda_\Qoppa \from \calC^\op \to \Sp$ which measures the difference between $\Qoppa$ and $\rmB_\Qoppa$ in a precise sense.}. We say a Hermitian category $(\calC,\Qoppa)$ is a \mdef{Poincar\'e category} if for each $X \in \calC$, $\rmB_\Qoppa(X,-) \from \calC^\op \to \Sp$ factors though the spectrum-valued Yoneda embedding so that we have $\rmB_\Qoppa(X,-) \simeq \map_\calC(-,\DQoppa(X))$ for a functor $\DQoppa \from \calC^\op \to \calC$. We require $\DQoppa$ is an equivalence (whose inverse is then necessarily $\DQoppa^\op$). We shall call the $\DQoppa$ the \mdef{duality} associated to $\Qoppa$. 
\end{defn}

\begin{rem}\label{rem: maps_hermitian_poincare_categories}
    Hermitian $\infty$-categories assemble into an $\infty$-category $\Cath$ of which Poincar\'e categories form a non-full subcategory $\Catp$. Explicitly $\Cath$ is the Cartesian unstraightening of the functor $\calC \mapsto \Funq(\calC^\op,\Sp)$. This means that a map of Hermitian categories $(\calC,\Qoppa_\calC) \to (\calD,\Qoppa_\calD)$ is a pair $(f,\eta)$ where $f \from \calC \to \calD$ is an exact functor between stable categories and $\eta \from \Qoppa_\calC \to f^* \Qoppa_\calD \defeq \Qoppa_\calC \circ (f^\op)$ is a natural transformation of functors $\calC^\op \to \Sp$. If $(\calC,\Qoppa_\calC)$ and $(\calD,\Qoppa_\calD)$ are Poincar\'e, then a map of Hermitian categories $(f,\eta)$ induces a map on cross effects $\rmB_{\Qoppa_\calC} \to (f \times f)^*\rmB_{\Qoppa_\calD}$. We obtain, for all $x \in \calC$, a natural transformation $\map_\calC(-,\rmD_{\Qoppa_\calC}(x)) \to \map_\calD(f(-), \rmD_{\Qoppa_\calD}(f^\op(x)))$ and by Yoneda a map in $f \rmD_{\Qoppa_\calC}(x) \to \rmD_{\Qoppa_\calD}f^\op(x)$. This map is natural in $x$ and so we obtain a natural transformation $f \circ \rmD_{\Qoppa_\calC} \to \rmD_{\Qoppa_\calD} \circ f^\op$. Then we say that $(f,\eta)$ is a map of Poincar\'e categories if the aforementioned natural transformation on dualities is an equivalence. Succinctly, maps of Poincar\'e categories are those maps of Hermitian categories which \mdef{preserve the duality}.
\end{rem}

We present some fundamental examples of Poincar\'e categories which will play a central role in this article.
\begin{exam}\label{exam: quadratic_symmetric}
    Let $\calC$ be a stable category with duality $\rmD$, that is $\rmD \from \calC^\op \xrightarrow{\simeq} \calC$ along with higher coherences thereby endowing $\calC$ with the structure of a homotopy fixed point of $\Catperf$ with the $(-)^\op$-action. To the pair $(\calC,\rmD)$ we associate a symmetric bilinear functor via the association $\rmB_\rmD \from (x,y) \mapsto \map_\calC(x, \rmD y)$. From this we build two canonical Poincar\'e categories with this cross effect.
    \begin{enumerate}
        \item $(\calC,\Qoppa_{\rmD}^q)$ with $\Qoppa_{\rmD}^q(x) = \rmB_\rmD(x,x)_{hC_2}$. This is the \mdef{quadratic} Poincar\'e structure associated to $\rmD$.
        \item $(\calC,\Qoppa_{\rmD}^s)$ with $\Qoppa_{\rmD}^s(x) = \rmB_\rmD(x,x)^{hC_2}$. This is the \mdef{symmetric} Poincar\'e structure associated to $\rmD$.
    \end{enumerate}
    $(\calC,\Qoppa_{\rmD}^q)$ and $(\calC,\Qoppa_{\rmD}^s)$ are respectively the initial and terminal Poincar\'e categories underlying stable category $\calC$ and duality $\rmD$. An important special case of these comes from \mdef{$\Eone$-rings with anti-involution}, that is, an $\Eone$-ring $A$ equipped with a ring map $\sigma \from A \xrightarrow{\simeq} A^\op$. In this case we obtain a duality on $\Perf(A)$ via $M \mapsto \map_A(M,A)$ with associated bilinear form  
    \[
    (M,N) \mapsto \map_{A \otimes A}(M \otimes N, A).
    \]
    We shall call a Poincar\'e structure $\Qoppa$ on $\Perf(A)$ \mdef{compatible with $\sigma$} if its duality (equivalently bilinear part) is as above.
\end{exam}

 Returning to our discussion of localising invariants, we may then define a \mdef{Poincar\'e localising invariant} as a functor $E \from \Catp \to \Sp$ sending bifibre sequences in $\Catp$ to exact sequences of spectra. Among those, the most fundamental are Grothendieck-Witt ($\GW$) theory and $\LL$-theory which are both universal in a precise sense - see \cite[Corollary 4.4.2, Theorem 4.4.12]{CDHII}. Heuristically, $\GW$-theory is at least as inaccessible as $\K$-theory and so is difficult to analyse directly. On the other hand $\LL$-theory is more easily understood than both $\K$ and $\GW$-theory.

  \subsection{Main Results}

This article takes inspiration from two papers, \cite{LMMT20} and \cite{Lan22}. The latter demonstrates a lack of redshift for $\LL$-theory of \textit{connective} $\Eone$-rings with anti-involution, see Corollary 15 in \textit{loc. cit.} for a precise statement. Further details can be found in Section \ref{section: chromatic_purity} of this article. The former paper proves a fundamental result concerning \textit{chromatic purity} of algebraic $\K$-theory. Broadly speaking, chromatic purity results for a localising invariant $E$ are those of the form ``for all rings $A$, some chromatic localisation of $E(A)$ depends only on some (possibly different) chromatic localisation of $A$''. Our goal will be to follow the structure of their proof to establish similar results for $\LL$-theory, and derive a number of interesting consequences. The statement of our main theorem is as follows.

\begin{thmx}[Purity]\label{thm: purity}\label{thm: A}
Let $(A,\sigma)$ be an $\Eone$-ring with anti-involution, and let $p=2$. 
\begin{enumerate}
    \item For $n\geq0$, the map $(A, \sigma) \to (\Lnf A, \Lnf \sigma)$ induces an equivalence on $T(n)$-local quadratic $\LL$-theory 
    \item For $n\geq1$, the map $(A, \sigma) \to (\LL_T A, \LL_T\sigma)$ induces an equivalence on $T(n)$-local quadratic $\LL$-theory. Here $T$ denotes $T(1)\oplus\cdots\oplus T(n)$.
\end{enumerate} 
\end{thmx}

In fact, part (2) of this theorem together with \cite[Theorem 1]{Lan22} implies that quadratic $\LL$-theory of \textit{any ring} vanishes $T(n)$-locally for any $n \geq 1$, since we may reduce to the connective case. This allows us to invoke Land's aforementioned result to show that quadratic $\LL$-theory does not exhibit chromatic redshift for all $\Eone$-rings.
\begin{thmx}[No Redshift]\label{thm: B}
     Suppose $n \geq 0$. Let $A$ be an $\Eone$-ring with anti-involution $\sigma$ and let $(\Perf(A),\Qoppa)$ be a Poincar\'e category with duality compatible with $\sigma$. Then if $A$ is $T(n)$-acyclic then $\LL(\Perf(A),\Qoppa)$ is $T(n)$-acyclic.
\end{thmx}
In particular, if $A$ has height $n \geq 0$ then $\LL^s(A,\sigma)$ has height $\leq n$. By this notation we mean $\LL^s(A,\sigma) \defeq \LL(A,\Qoppa_\sigma^s)$. We will use similar notation for the quadratic and Tate Poincar\'e structures.

Theorems \ref{thm: A} and \ref{thm: B} has an important consequence for chromatic vanishing results for general idempotent complete Poincar\'e categories by the following result.
\begin{prop}\label{prop: filtered_colimit_rings_with_involution}
    If $(\calC,\Qoppa)$ be an idempotent complete Poincar\'e category, then $(\calC,\Qoppa)$ is a filtered colimit of Poincar\'e categories associated to rings to involution.
\end{prop}
\begin{proof}
By \cite[Observation 1.2.19]{CDHI} we may write $(\calC,\Qoppa)$ as a filtered colimit over the poset of its full subcategories $\calC_i$ generated by finitely many elements and closed under the duality, where we give each full subcategory the restricted Poincar\'e structure. By idempotent completeness and cofinality, we may take this colimit over the poset $I$ consisting of subcategories $\calC_i$ generated by a single self dual object $x_i$ of the form $x_i = y_i \oplus \rmD(y_i)$ for some $y_i \in \calC$. Here, $x_i$ carries a canonical hyperbolic form. Furthermore, by the Schwede-Shipley theorem $\calC_i \simeq \Perf(\End(x_i))$ via $x  \in \calC_i \mapsto \map_{\calC}(x_i,x)$.  The hyperbolic form yields a point $q \in \Omega^\infty \rmB(x_i,x_i)^{hC_2}$ which by adjunction is a map $S^0 \to \rmB(x_i,x_i)^{hC_2}$, that is to say, a $C_2$-equivariant map of spectra $S^0 \to \rmB(x_i,x_i) = \map_\calC(x_i,D(x_i))$. 

    Under the Schwede-Shipley equivalence, this corresponds to a $C_2$-equivariant map of $\End(x_i)$-modules
    \[
    \End(x_i) \to \map_{\End(x_i)}(\End(x_i),D(\End(x_i))).
    \]
    The right hand side is the module with involution associated to the Poincar\'e category $(\Perf(\End(x_i)),\Qoppa^q)$. The map in question identifies with post composition with the given equivalence $x_i \to D(x_i)$ and is thus an equivalence. By the recognition criterion of \cite[Proposition 3.1.14]{CDHI}, the Poincar\'e structure on $\Perf(\End(x_i))$ comes from an anti-involution $\sigma_{i}$ on $\End(x_i)$.  
\end{proof}

Since $\LL$-theory and the $\LTn$ both commute with filtered colimits, we may leverage Proposition \ref{prop: filtered_colimit_rings_with_involution} to obtain a variety of more general results. First we have Theorem \ref{thm: A}(1) for categories.
\begin{thmx}\label{thm: A_1_categories}\label{thm: C}
    Let $(\calC,\Qoppa)$ be an idempotent complete Poincar\'e category. For $n\geq 0$, the map $(\calC, \sigma) \to (\Lnf \calC, \Lnf \Qoppa)$ in $\Catpidem$ induces an equivalence on $T(n)$-local quadratic $\LL$-theory.
\end{thmx}
\begin{proof}
    We recall that by \cite[Corollary 2.13]{BMCSY24}, the functor $(\calC,
    \Qoppa) \mapsto (\Lnf\calC,\Lnf\Qoppa)$ is given by tensoring with $\Perf(\Lnf S)$ and all idempotent Poincar\' categories of this form are closed under limits and colimits in $\Catpidem$. This observation, combined with Proposition \ref{prop: filtered_colimit_rings_with_involution} allows us to write  $(\Lnf\calC,\Lnf\Qoppa)$ as a filtered colimit in $\Catpidem$ of Poincar\'e categories associated to rings with involution, each of whose underlying stable category is of the form $\Perf(\Lnf A_i)$. The map
    \[
     \LTn\LL(\calC,
    \Qoppa) \mapsto \LTn \LL (\Lnf\calC,\Lnf\Qoppa)
    \]
    in $\Catpidem$ is induced by maps of the form $ \LTn\LL^q(A, \sigma) \to  \LTn \LL^q(\Lnf A, \Lnf \sigma)$ each of which is an equivalence and so the result follows.
\end{proof}


Another is the vanishing of higher chromatically localised quadratic $\LL$-theory for \textit{all} idempotent complete Poincar\'e categories.
\begin{thmx}\label{thm: D}
   Suppose $n \geq 1$. Let $\Qoppa$ be an a Poincar\'e structure on $\calC$ where $\calC$ is idempotent complete. Then $ \LL^q(\calC,\Qoppa)$ is $T(n)$-acyclic. 
\end{thmx}
\begin{proof}
   Passing to filtered colimits (and using the notation present in the proof of Proposition \ref{prop: filtered_colimit_rings_with_involution}) we are position to apply Theorem \ref{thm: B}. For $n \geq 1$ we have
    \[
     \LTn \LL^q(\calC,\Qoppa) \simeq \colim_{I} \LTn  \LL^q(\calC_i,\Qoppa|_{\calC_i})) \simeq \colim_{I}  \LL^q(\Perf(\End(x_i),\Qoppa_{\sigma_{i}}))
    \]
    where $\Qoppa_{\sigma_{i}}$ a Poincar\'e structure on $\Perf(\End(x_i))$ compatible with the involution $\sigma_{i}$. The latter vanishes term-wise, first by using Theorem \ref{thm: A} to reduce to the connective case, then by the truncating property of quadratic $\LL$-theory to reduce to the discrete case, and then finally by appealing \cite[Theorem 1]{Lan22} which shows vanishing from the discrete case.
    \end{proof}

The above results tells us that an invariant called \textit{normal $\LL$-theory} controls the behaviour of chromatically localised $\LL$-theory. For a definition, see Remark \ref{rem: normal_L_theory}.  
Given our results, is easy to deduce a version of the \textit{homotopy limit problem} for chromatically localised $\GW$-theory. We provide the proof here for the reader interested only in $\GW$-theory.
\begin{thmx}\label{thm: categorical_homotopy_limit}
    Let $n\geq 0$ and let $(\calC,\Qoppa)$ be a an idempotent complete Poincar\'e category whose underlying category has $T(n+1)$-acyclic endomorphism spectra, then
    \[
 L_{T(n+1)}\GW(\calC,\Qoppa) \simeq ( L_{T(n+1)}\K(\calC))^{hC_2}.
     \]
      In particular if $\calC$ is of the form $\Perf(A)$ and the given Poincar\'e structure endows $(\Perf(A),\Qoppa)$ with the structure of a symmetric monoidal Poincar\'e category, then $\GW(A,\Qoppa)$ has height $\leq n+1$.
\end{thmx}
\begin{proof}
   The  genuine $C_2$-spectrum $\KR(\calC,\Qoppa)$ of \cite[Corollary 4.5.2]{CDHII} gives rise to a a fibre sequence of the form
   \[
   \K(\calC,\Qoppa)_{hC_2} \to \GW(\calC,\Qoppa) \to \LL(\calC,\Qoppa).
   \]
   Applying $L_{T(n+1)}$, and using the fact that category of $T(n)$-local spectra is $\infty$-semiadditive from \cite{CSY20}, we may rewrite this sequence as 
    \[
   (L_{T(n+1)}\K(\calC,\Qoppa))^{hC_2} \to L_{T(n+1)}\GW(\calC,\Qoppa) \to L_{T(n+1)}\LL(\calC,\Qoppa).
   \]
   By Lemma \ref{cor: mapping_spectra_stable_L_theory}, the latter term vanishes. The final result follows from \cite[Theorem B]{LMMT20} and the existence of an $\Einf$-ring map $(L_{T(n+1)}\K(A))^{hC_2} \to L_{T(n+1)} \K(A)$.
\end{proof}
    From Theorem \ref{thm: categorical_homotopy_limit}, we see that that $T(n+1)$-local $\GW$-theory has the same descent properties as $T(n+1)$-$\K$-theory under the assumptions that mapping spectra are $\Lnf$-local. This recovers \cite[Theorem C]{CMNN20a} and \cite[Theorem 1.1]{BMCSY24} for $\GW$-theory. 
    \begin{thmx}
        Let $n\geq 0$ and let $X$ be a $\pi$-finite $2$-group acting idempotent complete Poincar\'e category $(\calC,\Qoppa)$ then we have
        \[
          L_{T(n+1)}\GW(\lim_X (\calC,\Qoppa))  \simeq \lim_X  L_{T(n+1)}\GW(\calC,\Qoppa).
        \]
        The colimit result follows analogously.
    \end{thmx}
    \begin{proof}
    The proof is a direct computation using the $\infty$-semiadditivity of the category of $T(n+1)$-local spectra and the fact that $L_{T(n+1)}\K(-)$ satisfies descent with respect to actions of $\pi$-finite spaces.
\[
      L_{T(n+1)}\GW(\lim_X (\calC,\Qoppa)) \simeq (L_{T(n+1)}\K(\lim_X (\calC,\Qoppa)))^{hC_2} \simeq (\lim_X L_{T(n+1)}\K( (\calC,\Qoppa)))^{hC_2} \simeq \lim_X L_{T(n+1)}\GW(\calC,\Qoppa)
    \] 
     and similarly for the colimit.
    \end{proof}
    It would be interesting to know whether $\LTn\LL(-)$ satisfies descent with respect to $\pi$-finite 2-groups in general. This is the subject on ongoing investigation.


We state one last theorem, analogous to chromatic redshift for $\K$-theory of $\Einf$-rings. This theorem rests upon the Hermitian trace methods and real topological Hochschild homology developed in the unpublished work of \cite{HNS22}
\begin{thmx}[Whiteshift for $\LL$-theory]\label{thm: whiteshift}
    Let $n \geq 0$. Let $A$ be an $\Einf$-ring with the trivial anti-involution $\id$. If $A$ is of height exactly $n$ then $\LL^t(A,\id)$ is of height exactly $n$.
\end{thmx}
\begin{proof}
    We have proven that $\height \LL^t(A) \leq \height A$ and are thus required to prove the other implication. By  \cite{HNS22}, there is a map of rings $\LL^n(A) \to A$ coming from the fact that they identify $A$ with $\Phi^{C_2}\THR(\Perf(A),\Qoppa_\id^t)$, the $C_2$-geometric fixed points of real topological Hoschschild homology. Thus $\LL_{T(n)}\LL^t(A)$ cannot vanish if $\LL_{T(n)}A$ does not.
\end{proof}
Here $\LL^t(-)$ is refers to the $\LL$-theory of the Tate Poincar\'e structure $\Qoppa^t_\sigma$ which only exists in the $\Einf$-setting, see \cite[Example 3.2.12]{CDHI}. 
\begin{warn}
    The whiteshift theorem cannot hold general Poincar\'e categories. By Theorem \ref{thm: C}, we will see that it almost never holds for the quadratic structure. Even without Theorem \ref{thm: C}, we can make use of the property of bordism invariance for $\LL$-theory to show that we can always find a ring spectrum for which whiteshift fails.  Consider the $\Eone$-ring $A \oplus A^\op$ equipped with the flip involution. The associative symmetric (and indeed also quadratic) Poincar\'e category is then simply $\Hyp(\Perf(A))$. A fundamental result of \cite{CDHII} is that $\LL$-theory of any Poincar\'e category of the form $\Hyp(\calC)$ is $0$. 
\end{warn}


\begin{ackn*}
  We would like to thank Yonatan Harpaz and Victor Saunier for their consistent insight and support while writing this article. In addition, we would like to thank Markus Land and Maxime Ramzi for many a fruitful discussion. This work was conducted under the auspicies of the European Research Council as part of the project Foundations of Motivic Real K-Theory (ERC grant no. 949583). 
\end{ackn*}

\begin{convs*}
By category we mean $\infty$-category unless otherwise stated.
For all chromatic localisations, we work exclusively at the prime $p = 2$ unless otherwise stated. In particular $\Lnf$ will denote $L_{n,2}^f$.Generally we will reserve the symbols $\calA$ and $\calB$ for (semi-)additive categories and $\calC$, $\calD$, and $\calE$ for stable categories.
\end{convs*}

\section{Additive Hermitian K-Theory}\label{section: additive_hermitian_poincare}
\subsection{Additive Hermitian and Poincar\'e Categories}
First we recall some standard definitions.

\begin{defn}\label{defn: additive_semiadditive}
    Let $\calA$ be a pointed category admitting finite coproducts. We say that $\calA$ is \mdef{semiadditive} if for each $a,b \in \calA$ the canonical maps
\[
\begin{tikzcd}
a & a \coprod b \arrow[l] \arrow[r] & b
\end{tikzcd}
\]
    exhibit $a \coprod b$ as the product of $a$ and $b$. In this case we write $a \oplus b$ for both the (canonically identified) product and coproduct. We say that $\calA$ is \mdef{additive} if it is semiadditive and every $a \in \calA$ the shear map $a \oplus a \to a \oplus a$ is an equivalence. Finally we say that $\calA$ is \mdef{$\flat$-additive} if it is additive and weakly idempotent complete - every map admitting a retract is an inclusion of a direct summand. The categories of semiadditive, additive, $\flat$-additive will be denoted $\Catsadd, \Catadd, \Cataddflat$ respectively. In these categories, functors are taken to be additive, that is, coproduct (equivalently product) preserving. 
\end{defn}
    There is a vast literature on semiadditive and additive categories, see \cite{GGN15} in particular for a treatment of the $\infty$-categorical case. In this paper they show that $\Catsadd$ (resp. $\Catadd$) is canonically enriched over $\CMon$ (resp. $\CGrp$). There is comparatively little in the literature about $\flat$-additive categories. One notable property is that the category $\flat$-additive categories is is equivalent to the category of weighted stable categories. We discuss Poincar\'e categories in the $\flat$-additive case, where we follow the treatment in \cite{HS21}.
\begin{defn}\label{defn: additive_hermitian_poincare}
    Let $\calA$ be an $\flat$-additive $\infty$-category.  We will say that a reduced functor $\Qoppa \from \calA^\op \to \CGrp$ is \mdef{additive quadratic} if the cross effect $\rmB_\Qoppa \from \calA^\op \times \calA^\op \to \CGrp $ defined by the formula
    \[
    \Qoppa(x \oplus y)  \simeq \Qoppa(x) \oplus \Qoppa(y) \oplus \rmB_\Qoppa(x,y)
    \]
    is \mdef{additive bilinear}, that is, preserves direct sums in each variable separately. We denote by $\Funaq(\calA^\op, \CGrp)$ the full subcategory of $\Fun(\calA^\op, \CGrp)$ consisting of additive quadratic functors. A \mdef{$\flat$-additive Hermitian category} is a pair $(\calA,\Qoppa)$ with $\calA$ $\flat$-additive and $\Qoppa$ an additive quadratic functor. The category of such is the Cartesian unstraightening of the functor
    \[
    (\Cataddflat)^\op \to \Cat, \quad \calA \mapsto \Funaq(\calA^\mathrm{op}, \CGrp).
    \]
    We call a morphism in this category an additive Hermitian functor.
     The category of \mdef{$\flat$-additive Poincar\'e categories} $\Catap$ is the subcategory $\Catah$ consisting of pairs $(\calA, \Qoppa)$ where $\rmB_\Qoppa(x,y) \simeq \Map_\calA(a,\DQoppa(b))$ naturally in $a,b \in \calA$ for an equivalence $\DQoppa \from \calA^\mathrm{op} \to \calA$. We require maps between $\flat$-additive Poincar\'e categories to be duality preserving in the sense of Remark \ref{rem: maps_hermitian_poincare_categories}.
 
\end{defn}
From the above, the data of a Hermitian functor $(\calA,\Qoppa) \to (\calA',\Qoppa')$ is the data of an additive functor $f\from \calA \to \calA'$ and a natural transformation $\eta \from \Qoppa \to f^* \Qoppa' \defeq \Qoppa' \circ f^\op$. As with the stable case we have two basic formulas for $\rmB_\Qoppa$, which hold irrespective of any further assumptions on $(\calA,\Qoppa)$.
    \begin{align*}
           \rmB_\Qoppa(-,-) &\simeq \fib \left (\Qoppa(- \oplus -) \to \Qoppa(\pi_1(-,-)) \oplus \Qoppa(\pi_2(-,-))\right);\\
            \rmB_\Qoppa(-,-) &\simeq \cofib \left(\Qoppa(\pi_1(-,-)) \oplus \Qoppa(\pi_2(-,-)) \to \Qoppa( - \oplus - ) \right).        
    \end{align*}
It follows that limits and colimits in $\Funaq(\calA)$ exist and can be computed in $\Fun(\calA^\op,\CGrp)$. $\Catap$ enjoys essentially the same functoriality properties as $\Catp$ and we will sketch a number of these results in Appendix \ref{appendix: functoriality}.
If $(\calC,\Qoppa)$ be a stable Poincar\'e category, then $(\calC,\loopsinf \Qoppa)$ is $\flat$-additive Poincar\'e category.
Motivated by this we supply the following definition:
\begin{defn}\label{defn: underlying_additive}
    Let $(\calC,\Qoppa)$ be a stable Poincar\'e category. We shall call the $\flat$-additive Poincar\'e category $(\calC,\loopsinf\Qoppa)$ the \mdef{underlying $\flat$-additive Poincar\'e category} of $(\C,\Qoppa)$ and will be denoted $\flatAdd(\C,\Qoppa)$.
\end{defn}
There is a construction in the opposite direction.
\begin{defn}
    Let $(\A,\Qoppa)$ be a $\flat$-additive Poicar\'e category. Then we define the \mdef{associated stable Poincar\'e category} $\Stab(\A,\Qoppa)$ to have underlying stable category $\Stab(A)$ with Poincar\'e structure given by the unique quadratic extension of $\Qoppa$ to $\Stab(\calA)$. This quadratic extension comes from \cite[Theorem 2.19]{BGMN22} applied to the case $n=2$.
\end{defn}

We come to the fundamental example of an additive Poincar\'e category. Recall from Proposition \ref{prop: additive_hermitian_schwede-shipley} that a Hermitian structure on $\Proj^\omega(R)$ gives rise to module with $C_2$-action $M$ encoding the bilinear part. The next result gives us a necessarily and sufficient criterion for aforementioned Hermitian category to be Poincar\'e. This is the analogue of \cite[Proposition 3.1.6]{CDHI} with the same proof. Recall the notation of an invertible $A$-module with involution from \cite[Definition 3.1.4]{CDHI}, where the reader can find further discussion.
\begin{defn}
    Let $M \in \Mod_{A \otimes A}$. Viewing $M$ as an $A$-module via the first action, we obtain grouplike $\Einf$-space $\Map_A(M,M)$. This inherits an $A$-action via the residual action on $M$ from which we obtain a map of grouplike $\Einf$-spaces
    \[
    A \mapsto \Map_A(M,M).
    \]
    We say that $M$ is \mdef{invertible} if the above map is an equivalence.
\end{defn}
\begin{cor}\label{prop: additive_poincare_schwede-shipley}
 Suppose $(\Proj^\omega(A), \Qoppa_M)$ is additive Hermitian with underlying biadditive functor $\rmB_{\Qoppa_M}(X,Y) \simeq \Map_{A \otimes A}(X \otimes Y, M)$. Then $(\Proj^\omega(A), \Qoppa_M)$ is Poincar\'e if and only if $M$ is invertible and compact projective and the associated duality is given by $X \mapsto \Map_A(X,M)$.
\end{cor}
\begin{proof}
    Fix $Y \in \Proj^\omega(A)$. The functor $X \in \Proj^\omega(A)^\op \mapsto \Map_{A \otimes A}(X \otimes_A Y, M) \in \CGrp$ is represented in $\Mod(A)$ by (the a priori not compact projective) $\Map_A(Y,M)$ with $A$-action induced by the second $A$-action on $M$. By idempotent completeness, for this to be compact projective for all $Y$ it is necessary and sufficient to demand that $\Map_A(A,M) \simeq M$ is compact projective. We denote the functor $N \mapsto \Map_A(N,M)$ by $\rmD_M$. Now if $M$ is compact projective, since $\Proj(R)$ is generated by $R$ it is necessary and sufficient to require that the the canonical map $X \mapsto \rmD_M \rmD_M (X) = \Map_A(\Map_A(X,M),M)$ is an equivalence only in the special case $X = A$. This is an equivalence if and only if $M$ is in addition invertible.
\end{proof}

\subsection{Additive GW and L-theory}

Before defining additive analogues of $\GW$ and $\LL$-theory, we will need to discuss the additive analogous of the constructions which appear in \cite[Section 2.1]{CDHI}. In particular, we need to discuss the space of Poincar\'e  (Hermitian)  objects associated to a $\flat$-additive Poincar\'e (Hermitian) category $(\A,\Qoppa)$. Most of the definitions will go through verbatim.

\begin{defn}\label{defn: hermitian_forms}
    Let $(\A,\Qoppa)$ be a $\flat$-additive Hermitian category. 
         The \mdef{category of Hermitian objects}  in $(\A,\Qoppa)$ is given by the total category of the right fibration 
    \[
    \He(\A,\Qoppa) \defeq \int_{a \in \A} \Qoppa(a) \to \A
    \]
    classifying the functor $\Qoppa$. Thus, an object of $\He(\A,\Qoppa)$ consists of a pair of objects $(a,q) \in \calA \times \Qoppa(a)$. We will denoted by $\Fm(\A,\Qoppa) \subseteq \He(\A,\Qoppa)$ the core subgroup of $\He(\A,\Qoppa)$ and refer to it as the \mdef{space of Hermitian objects}  in $(\A,\Qoppa)$.  Finally, if $(\A,\Qoppa)$ is Poincar\'e, we let $\Pn(\A,\Qoppa)$ denote the \mdef{space of Poincar\'e objects}  in $(\A,\Qoppa)$, defined to be the full subgroupoid of $\Fm(\A,\Qoppa)$ spanned by those objects $(a,q)$ for which the canonical map $q_\# \from a \to \DQoppa (a)$ is an equivalence.  
\end{defn}

\begin{lem}\label{lem: pn_filtered_colimits}
    The functors $\Fm \from \Catah \to \An$ and $\Pn \from \Catap \to \An$ commute with filtered colimits of additive Poincar\'e categories.
\end{lem}
\begin{proof}
    The proof is entirely analogous to \cite[Proposition 6.1.8]{CDHI}. The key ingredient of their proof, that $\Catex \to \Cat$ preserves filtered colimits in our case becomes the statement that $\Catadd \to \Cat$ preserves filtered colimits.  To see this, note that $\Catadd \to \Cat$ can be factored as $ \Catadd \to \Catsadd \simeq \Mod_{\Span_0}(\Catcoprod) \to \Catcoprod \to \Cat$. Each functor in this chain preserves filtered colimits. For the first we observe group-like condition on mapping spaces is closed under colimits. For the second we have that $\Mod_{\Span_0}(\Catcoprod) \to \Catcoprod$ admits a right adjoint (and indeed also a left adjoint) \cite{Har20}. Finally for the third we have the general fact that $\Catcoprod \to \Cat$ preserves sifted colimits by the following argument, due to Maxime Ramzi.
    Suppose $\C_{(-)} \from I \to \Catcoprod$ is sifted diagram, with $\C \defeq \colim_I \C_i$ in $\Cat$. It suffices to prove the following three claims:
    \begin{enumerate}
        \item For each $x,y \in \C$, there exists an $i \in I$ along with $x_i,y_i \in \C_i$ such that the canonical map $\C_i \to \calC$ maps $x_i$ and $y_i$ to $x$ and $y$, respectively.
        \item $\calC$ has finite coproducts. 
        \item For each $i \in I$, $\C_i \to \calC$ preserves finite coproducts.
    \end{enumerate}
The proof of the claims are given below:
        \begin{enumerate}
            \item[(1)] We calculate $\C \times \C \simeq \colim_{(i,j) \in I \times I} \C_i \times \C_j \simeq \colim_{i \in I } \C_i \times \C_i$, where the latter equivalence follows from siftedness. $\Delta^0$ is compact projective in $\Cat$, which implies the claim 
            \item[(2)+(3)] Since each $\C_i$ has finite coproducts, each diagonal functor
            $\Delta_i \from \C_i \to \C_i \times \C_i$ has a left adjoint, namely the binary coproduct. In the map of diagrams $\C_i \to \C_i \times \C_i$ induced by the diagonal, each naturality square is left adjointable since each $\C_i \to \C_j$ is coproduct preserving. Hence upon taking colimits, the resulting map $\C \to \colim_{i \in I} (\C_i \times \C_i)$ admits a left adjoint compatible with each $\C_i \to \colim_{i \in I} (\C_i \times \C_i)$. By siftedness,  $\colim_{i \in I} (\C_i \times \C_i)$ identifies with $\C \times \C$ and so $\C$ has finite coproducts and each $\C_i 
            \to \C \times \C$ preserves these adjoints, that is to say each $\C_i 
            \to \C \times \C$ is coproduct preserving.
        \end{enumerate}
\end{proof}

With the definitions of $\flat$-additive Poincar\'e categories at hand, we continue to follow \cite[Appendix A]{HS21} for the corresponding suitable definitions of additive $\GWspace$ and $\Lspace$. 

\begin{defn}\label{defn: additive_GW}
The \mdef{additive $\GW$-space} associated to $(\calA,\Qoppa)$, $\GWspace^\oplus(\calA,\Qoppa)$ is given by
\[
\GWspace^\oplus(\calA,\Qoppa) \defeq \Pn(\calA,\Qoppa)^\gp
\]
\end{defn}

This manifestly induces a functor $\GWspace^\oplus \from \Catap \to \CGrp$. This functor preserves filtered colimits in $\Catap$ because both $\Pn$ and $(-)^\gp$ do. The reader might be concerned that this definition is too naive, given than the stable variant is rather involved. In fact, the the group completion within the category $\Fun^{\oplus}(\Catap,\CMon)$ coincides with the pointwise group completion, a phenomenon which one sees when comparing additive and stable $\K$-theory. More involved is the definition of the additive $\LL$-space functor. For a fixed $\flat$-additive Poincar\'e category $(\A,\Qoppa)$, one has, for each $[n] \in \Delta$, an additive Hermitian $\mathrm{Q}$-construction, denoted $\mathrm{Q}_{n}(\A,\Qoppa)$. As $n$ varies, we obtain a simplicial object in $\Catap$ denoted $\mathrm{Q}_\bullet(\A,\Qoppa)$. For any functor $\mathcal F \from \Catap \to \CGrp$, we then define 
\[
\mathcal F\mathrm{Q}^{(n)}(\A,\Qoppa) \defeq \colim_{([k_1],\ldots,[k_n]) \in \Delta^{\times n}} \mathcal{F} \left( \mathrm{Q}_{k_1} \cdots \mathrm{Q}_{k_n}(\A,\Qoppa) \right).
\]
 In the above formula, the object $\mathcal{F} \left( \mathrm{Q}_{k_1} \cdots \mathrm{Q}_{k_n}(\A,\Qoppa) \right)$ refines to an $n$-fold simplicial object in $\CGrp$, that is, a functor $(\Delta^\op)^{\times n} \to \CGrp$. This colimit amounts to taking the $n$-fold geometric realisation, so that $\mathcal F\mathrm{Q}^{(n)}(\A,\Qoppa) \in \CGrp$.
 
\begin{defn}\label{defn: additive_L}
Let $(\A,\Qoppa)$ be a $\flat$-additive Poincar\'e category and consider $\GWspace^\oplus \mathrm{Q}^{(n)}(\A,\Qoppa)$. We define the \mdef{additive $\LL$-theory space} by
\[
\Lspace^\oplus(\calA,\Qoppa) \defeq \colim_{n \in \bbN} \GWspace^\oplus\mathrm{Q}^{(n)}(\A,\Qoppa).
\]
The above colimit is sequential, induced by the inclusions $\Delta^{\times \, (n-1)} \times \{[0]\} \to \Delta^{\times \, n}$.
\end{defn}

 According to \cite[Proposition A.2.6]{HS21}, additive $\LL$-theory fits into a bifibre sequence analogous to the fundamental fibre sequence of \cite{CDHII}. 
\begin{thm}\label{thm: fundamental_additive_bifibre_sequence}
  The additive $\LL$-theory functor $\Lspace^\oplus$ is the cofibre of the natural map $\Kspace^\oplus(-)_{\mathrm{hC}_2} \to \GWspace^\oplus(-)$, taken in $\Fun(\Catap,\CGrp)$. This is also a fibre sequence.
\end{thm}

An immediate corollary of this theorem is the following.
\begin{cor}\label{cor: additive_L_colimits}
    Let $(\A_i,\Qoppa_i)$ be filtered diagram of $\flat$-additive Poincar\'e categories admitting a colimit $(\A,\Qoppa)$. Then 
    \[
    \Lspace^\oplus (\A,\Qoppa) \simeq \colim \Lspace^\oplus(\calA_i,\Qoppa_i).
    \]
\end{cor}
\begin{proof}
    $\Kspace^\oplus$ commutes with filtered colimits and thus so does its homotopy orbits. $\GWspace^\oplus$ commutes with filtered colimits since $\Pn$ (by Lemma \ref{lem: pn_filtered_colimits}) and pointwise group completion do. Therefore $\Lspace^\oplus$ commutes with filtered colimits.
\end{proof}

We have two key theorems which relate additive to stable $\GW$ and $\LL$-theory.
\begin{thm}{\cite[Theorem B,Appendix A.1.2]{HS21}}
    Let $(\A, \Qoppa)$ be a $\flat$-additive Poincar\'e category. Then there are natural equivalences
    \begin{enumerate}
     \item  $\GWspace^\oplus(\A,\Qoppa) \simeq \GWspace(\Stab(\A,\Qoppa))$;
    \item   $ \Lspace^\oplus(\A,\Qoppa) \simeq \Lspace(\Stab(\A,\Qoppa))$.  
    \end{enumerate}
\end{thm}

Working backwards, we have formula for the stable $\GW$-space in terms of $\GW^\oplus$.
\begin{lem}\label{lem: description_additive_GW}
    For a stable Poincar\'e category $(\calC,\Qoppa)$, we have the identification of commutative groups 
    \[
    \Omega^{\infty - 1}\GW(\calC,\Qoppa) \simeq |\GWspace^\oplus \flatAdd\Qdot(\calC,\Qoppa^{[1]})|
    \] 
    functorial in $(\calC,\Qoppa)$.
\end{lem}
\begin{proof}
We begin with the formula
\[
 \Omega^{\infty - 1}\GW (\calC,\Qoppa) \simeq |\Pn \Qdot(\calC,\Qoppa^{[1]})|\] 
of \cite[Corollary 4.2.3]{CDHII} which is functorial in $(\calC,\Qoppa)$. The functor $\Pn$ is invariant under taking the underlying $\flat$-additive Poincar\'e category and so we may replace $\Qdot(\calC,\Qoppa^{[1]})$ with $\flatAdd\left(\Qdot(\calC,\Qoppa^{[1]})\right)$. Performing this identification obtain 
\[
 \Omega^{\infty - 1}\GW  (\calC,\Qoppa) \simeq |\Pn \flatAdd \Qdot(\calC,\Qoppa^{[1]})|
\]
 as objects of $\CMon$. On the right hand side we have taking the geometric realization in $\CMon$. The monoid structure on both sides are induced from the semiadditive structure of $\Catp$. The left hand side is group-like, essentially by definition. Thus the same is true of $|\Pn \flatAdd \Qdot(\calC,\Qoppa^{[1]})|$ and so it may be identified with $|\Pn \flatAdd  \Qdot(\calC,\Qoppa^{[1]})|^\gp$. Moreover, since group completion is a left adjoint the right hand side can be further identified with $|\Pn^\gp \flatAdd\Qdot(\calC,\Qoppa^{[1]})|$ where the geometric realization is now taken in $\CGrp$.  But this is precisely $|\GWspace^\oplus \flatAdd \Qdot(\calC,\Qoppa^{[1]})|$ and the lemma follows.
\end{proof}
The statement $\Omega^{\infty - 1}\GW  (\calC,\Qoppa) \simeq |\Pn \flatAdd \Qdot(\calC,\Qoppa^{[1]})|$ remains true if we regard the geometric realization as being taken in spaces since $\CMon \to \Spaces$ creates geometric realizations.
As a corollary of the previous result, we deduce a similar statement for $\LL$-theory. 
\begin{cor}\label{cor: description_additive_L}
    For a stable Poincar\'e category $(\calC,\Qoppa)$, we have the identification of commutative groups
    \[
    \Omega^{\infty-1}\LL (\calC,\Qoppa) \simeq |\Lspace^\oplus \flatAdd \Qdot(\calC,\Qoppa^{[1]})|
    \] 
    functorial in $(\C,\Qoppa)$.
\end{cor}
\begin{proof}
    The main result of \cite{CDHII} implies the existence of a cofibre sequence of spectra,
\[
\K (\calC,\Qoppa)_{\mathrm{hC}_2} \to \GW (\calC,\Qoppa) \to \LL (\calC,\Qoppa).
\]
By the explicit description of the $\LL_0$ in terms of $\GW_0$, this map is a surjection on $\pi_0$. Therefore have a cofibre sequence upon taking $\Omega^\infty$ or indeed $\Omega^{\infty-1}$. Thus we obtain
\[
\Omega^{\infty - 1} \left( \K (\calC,\Qoppa)_{\mathrm{hC}_2}\right) \to \Omega^{\infty-1}\GW (\calC,\Qoppa) \to \Omega^{\infty-1}\LL (\calC,\Qoppa)
\]
is a cofibre sequence in $\CGrp$.
To finish the argument, we consider the diagram of cofibre sequences, the bottom one coming from Theorem \ref{thm: fundamental_additive_bifibre_sequence}.
\[
\begin{tikzcd}
{\Omega^{\infty - 1} \left( \K (\calC,\Qoppa)_{\mathrm{hC}_2}\right) }\arrow[d] \arrow[r] & {\Omega^{\infty-1}\GW (\calC,\Qoppa) } \arrow[d] \arrow[r] & {\Omega^{\infty-1}\LL (\calC,\Qoppa)} \arrow[d] \\
{|\Kspace^\oplus (\flatAdd\Qdot(\calC,\Qoppa^{[1]}))_{\mathrm{hC}_2}|} \arrow[r]                           & {|\GWspace^\oplus  \flatAdd \Qdot(\calC,\Qoppa^{[1]})|} \arrow[r]                        & {|\Lspace^\oplus  \flatAdd\Qdot(\calC,\Qoppa^{[1]})|}.     
\end{tikzcd}
\]
We have already proven that the middle vertical arrow is an equivalence in Lemma \ref{lem: description_additive_GW}. Hence it suffices to show that the left vertical arrow is too. Consider the commutative diagram
\[
\begin{tikzcd}
(\Spcn)^{BC_2} \arrow[d, "(-)_{hC_2}"] \arrow[r, "\Omega^{\infty -1}_*"] & \CGrp^{BC_2} \arrow[d, "(-)_{hC_2}"] \\
\Spcn \arrow[r, "\Omega^{\infty-1}"]                                      & \CGrp                               
\end{tikzcd}.
\]
 Tracing around one direction gives  $\Omega^{\infty-1}\left( \K (\calC,\Qoppa)_{\mathrm{hC}_2}\right)$, while the other gives ${|\Kspace^\oplus (\flatAdd\Qdot(\calC,\Qoppa^{[1]}))_{\mathrm{hC}_2}|}$.
\end{proof}

\section{Applications to Chromatically Localised Hermitian K-theory}\label{section: chromatic_purity}
Let us begin by discussing the ingredients which go into the proof of Theorem \ref{thm: B}, recalled below.

\begin{mythmx}{B}
     Suppose $n \geq 0$. Let $A$ be an $\Eone$-ring with anti-involution $\sigma$ and Let $(\Perf(A),\Qoppa)$ be a Poincar\'e category with duality compatible with $\sigma$. Then if $A$ has height $\leq n$ then $\LL(\Perf(A),\Qoppa)$ has height $\leq n$.
\end{mythmx}

The first major ingredient we use is as yet unpublished work of Harpaz, Nikolaus, and Shah \cite{HNS22} which provides novel foundations for trace methods in Hermitian K-theory, including the following important result.
\begin{thm}\label{thm: magic_formula}
Suppose $A$ is an $\Eone$-ring. Let $\Qoppa$ be a Poincar\'e structure on $\Perf(A)$ with underlying module of involution $M$. Then the cofibre of the canonical map $\LL^q(A,M) \to \LL(A,\Qoppa)$, denoted \mdef{normal $\LL$-theory} of $(\Perf(A),\Qoppa)$, is a filtered colimit of spectra of the form
\[
\Eq\left(\map_A(D_M(X),X) \rightrightarrows (\bbS^{1 - \sigma} \otimes \map_A(D_M(X),X)_\hCtwo\right) 
\]
for $X$ a compact $A$-module.
\end{thm}
\begin{rem}\label{rem: normal_L_theory}
    Normal $\LL$-theory can be defined for arbitrary Poincar\'e categories analogously. Given $(\calC,\Qoppa)$ with duality $\rmD$, normal $\LL$-theory is defined to fit into a cofibre sequence 
    \[
    \LL^q(\calC,\rmD) \to \LL(\calC,\Qoppa) \to \LL^\mathrm{nor}(\calC,\Qoppa).
    \]
\end{rem}
From Theorem \ref{thm: magic_formula} if we assume at $A$ is $T(n)$-acyclic, a thick subcategory argument shows that the mapping spectra in the formula above are also $T(n)$-acyclic and hence the equaliser is too. We deduce
\begin{cor}[{\cite[Theorem 13]{Lan22}}]\label{cor: L_quadratic_Tn}
    Let $n \geq 0$. If $A$ is $T(n)$-acyclic, the canonical map $\LL^q(A,M) \to \LL(A,\Qoppa)$ is a $T(n)$-local equivalence.
\end{cor}
This result tells us that $T(n)$-local $\LL$-theory of rings of $T(n)$-acyclic rings is independent of the linear part of the quadratic structure. We shall use this fact extensively in this article and in future work. Land goes on to prove a no redshift theorem for \textit{connective} rings. He deduces it first for ordinary $\Eone$-rings using geometric methods and then finally for all connective rings using the truncating property of quadratic $\LL$-theory. Theorem \ref{thm: B} is, to our knowledge, the first result for general $\Eone$-rings.

For \textit{commutative} rings with anti-involution, the lack of redshift follows easily from Land's result as follows. There is a map of $\Einf$-rings
\[
\LL^s(\tau_{\geq0}A,\tau_{\geq 0}\sigma) \to \LL^s( A, \sigma).
\]
If $A$ is $T(n)$-acyclic then so is $\tau_{\geq0}A$. This forces $\LL^s(\tau_{\geq 0} A, \sigma)$ to vanish $T(n)$-locally. Since the result is true for symmetric $\LL$-theory, it is true for quadratic $\LL$-theory by Corollary \ref{cor: L_quadratic_Tn} and thus for any other Poincar\'e structure compatible with $\sigma$.

As remarked in the introduction, we will focus on proving chromatic purity for $\LL$-theory, thereby extending the results of \cite{LMMT20} to L-theory. Recall that by purity we mean any result of the form ``a chromatic localisation of $E(A)$ for a localising invariant $E$ and ring $A$ depends only on a (possibly different) chromatic localisation of $A$''. In \cite{LMMT20}, the following theorem is proved
\begin{thm}[{\cite{LMMT20}}]\label{thm: LMMT_purity}
    Let $A$ be a $\Eone$-ring then
    \begin{enumerate}
        \item For $n \geq 1$, the canonical map $A \to \Lnf A$ induces an equivalence on $\LTn \K (-)$ .
        \item For $n \geq 2$, the canonical map $A \to L_{T}$ induces an equivalence on $\LTn \K (-)$  Here $T = T(1) \oplus \cdots\oplus T(n)$.
    \end{enumerate}
\end{thm}
To prove this, the authors deduce properties of chromatically localised algebraic K-theory by making use of additive, rather than stable, K-theory. We will see that a similar proof strategy for L-theory. 

Their work, along with that of many others, has allowed for the use of trace methods to attack problems in chromatically localised algebraic $\K$-theory and resolve long-standing conjectures at the intersection of $\K$-theory and chromatic homotopy theory. We intend to revisit this theme in the context of Hermitian $\K$-theory in future works.

Below is the broad strategy used to prove Theorem \ref{thm: A} and Theorem \ref{thm: B}.
\begin{proof}[Proof Strategy for the ``Purity'' and ``No Redshift'' Theorems]\let\qed\relax
\hfill
    \begin{enumerate}
    \item\label{point1} Using additive $\LL$-theory, we show that any stable category with $\Lnf$-acyclic endomorphism spectra has vanishing $\LTi$-local stable $\LL$-theory for $0 \leq i \leq n$.
    \item\label{point2} A connectivity argument allows us to identify $\LTi \KaroubiL(-)$ with $\LTi\LL(-)$.
    \item\label{point3} We show the existence of a Hermitian chromatic fracture square for rings with involution.
    \item\label{point4} We deduce that $A \to \Lnf A$ induces an equivalence on $\LTi \LL^q (-)$ for rings with involution.
    \item\label{point5} Finally, using the Hermitian chromatic fracture square, we show that $A \to L_{T(1) \oplus \cdots \oplus T(n)} A$ induces an equivalence on $L_{T(1) \oplus \cdots \oplus T(n)} \LL^q(-)$ in a suitable range.
    \item\label{point6} Since  $L_{T(1) \oplus \cdots \oplus T(n)}$ is a periodic localisation, the problem is reduced to the connective case and we conclude from Land's results.
\end{enumerate}
\end{proof}

We begin by establishing Item \ref{point1}, the $\LL$-theoretic analogue of \cite[Proposition 3.6 (i)]{LMMT20}. Recall the following fact from \cite[Section 3.3]{LMMT20} that for stable $\calC$, $\map_\calC(X,X)$ being $\Lnf$-acyclic is equivalent to $\Map_\calC(X,X)$ being $\Lnf$-acyclic. Thus in this case we do not need to distinguish between mapping endomorphism spaces and mapping endomorphism spectra as far as $\Lnf$-acyclicity is concerned. Similarly for $T(i)$-acyclicity for $i \geq 0$. This important but surprisingly subtle point is further discussed in \textit{loc. cit.} 

\begin{prop}\label{prop: mapping_spaces_additive_L_theory}
Let $i \geq 1$. Suppose $(\calA,\Qoppa)$ is an  $\flat$-additive Poincar\'e category. If for each object $x \in \calA$, the endomorphism space $\Map_\calA(x,x)$, considered as connective spectrum, is $\Lnf$-acyclic, then $\Kspace^\oplus(\calA,\Qoppa)_{\hCtwo}$, $\GWspace^\oplus(\calA,\Qoppa)$, and $\Lspace(\calA,\Qoppa)$ are all $T(i)$-acyclic for $1 \leq i \leq n$.
\end{prop}
\begin{proof}
  First notice that the reduction step in the proof of \cite[Proposition 3.6 (i)]{LMMT20} goes through for $\LL$-theory with little modification. More explicitly, using Lemma \ref{lem: addititve_poincare_infty_cat_generated_finite_subcats} and the fact that $\Lspace^\oplus$ and $\GWspace^\oplus$ both commute with filtered colimits, we may reduce to the case where $(\calA,\Qoppa)$ is generated by a single object $X = x \oplus \DQoppa(X)$. In this case, the additive Schwede-Shipley theorem implies that $(\calA,\Qoppa)$ is an additive Poincar\'e category with underlying $\flat$-additive $\Proj^\omega\left(\Map_\calA(X,X)\right)$, equipped with a module of involution $M$. From \cite[A.2.6]{HS21appendix} we obtain a bifibre sequence of commutative groups
  \[
  \Kspace^\oplus(\Proj^\omega(\Map_\calA(X,X)),\Qoppa)_{\hCtwo} \to \GWspace^\oplus(\Map_\calA(X,X),\Qoppa) \to \Lspace^\oplus(\Map_\calA(X,X),\Qoppa).
  \]
  By \cite[Proposition 3.6(i)]{LMMT20}, the first term vanishes after $T(i)$-localisation for $1 \leq i \leq n$. Thus we get an identification 
  \[
  \LTi\GWspace^\oplus(\Map_\calA(X,X),\Qoppa) \xrightarrow{\simeq} \LTi\Lspace^\oplus(\Map_\calA(X,X),\Qoppa).
  \]
  For the other vanishing statements, we work exclusively with $\LL$-theory. The additive $\LL$-theory term is equivalent to $\LTi\LL(\Map_\calA(X,X),\Stab(\Qoppa))$. Since $\Map(X,X)$ is connective and $T(i)$-acyclic for $1 \leq i \leq n$, the same is true of $\LL$-theory by \cite[Corollary 15]{Lan22}. Not that the vanishing result for $i = 1$ is guaranteed by the assumption $p=2$. 
\end{proof}
Continuing with our proof strategy, we exploit the interplay of additive and Hermitian $\LL$-theory, we deduce the stable analogue of the above proposition in the manner of \cite[Proposition 3.6 (ii)]{LMMT20}.
\begin{prop}\label{prop: mapping_spectra_stable_L_theory}
  Let $i \geq 1$. Suppose $(\calC,\Qoppa)$ is a stable Poincar\'e category. If for each object $x \in \calC$ the endomorphism spectrum $\map_\calA(x,x)$ is $\Lnf$-acyclic, then $\K(\calC,\Qoppa)_{\hCtwo}$, $\GW(\calC,\Qoppa)$, and $\LL(\calC,\Qoppa)$ are all $T(i)$-acyclic for $1 \leq i \leq n$.
\end{prop}
\begin{proof}
  Immediately from \cite[Proposition 3.6(ii)]{LMMT20} we deduce that $T(i)$-local $\K$-theory vanishes for $1 \leq i \leq n$. Hence from the fundamental fibre sequence for stable Hermitian $\K$-theory
  \[
  \K(\calC)_{\hCtwo} \to \GW(\calC,\Qoppa) \to \LL(\calC,\Qoppa),
  \]
   we obtain the desired equivalence of $T(i)$-local $\GW$ and $\LL$-theory. We $T(i)$-localise the formula of Lemma \ref{lem: description_additive_GW}
     \[
    \LTi\Omega^{\infty-1}\GW(\calC,\Qoppa) \simeq |\LTi\GWspace^\oplus \flatAdd\Qdot(\calC,\Qoppa^{[1]})|.
    \] 
 The vanishing statements will follow from Proposition \ref{prop: mapping_spaces_additive_L_theory} provided we can prove that the right hand side is $T(i)$-acyclic for $1 \leq i \leq n$. For this, recall that the underlying stable category of the Hermitian Q-construction is the certain full subcategory of $\Fun(\TwAr(\Delta^k),\calC)$. It therefore suffices to show that for each $k$, all endomorphism spaces (equivalently spectra) of $\Fun(\TwAr(\Delta^k),\calC)$ are $T(i)$-acyclic.  We may write these endomorphism spectra as the following finite limit
\[
\map_{\Fun(\TwAr(\Delta^k),\calC)}(F,F) = \lim_{[i \to j] \in \TwAr(\TwAr(\Delta^k))^\op} \map_\calC(F(i),F(j)).
\]
Each mapping spectrum $\map_\calC(F(i),F(j))$ is a module over $ \map_\calC(F(i),F(i))$ and is thus $T(i)$-acyclic. Since the class of $T(i)$-acyclic spectra is closed under finite limits, the result follows. 
\end{proof}

If we are content to only have results for $\LL$-theory, we can make do with less restrictive assumptions with essentially the same proof. 
\begin{prop}
Let $n \geq 0$. Suppose $(\calA,\Qoppa)$ is an $\flat$-additive Poincar\'e category. If for each object $x \in \calA$, the endomorphism space $\Map_\calA(x,x)$, considered as a connective spectrum, is $T(n)$-acyclic, then $\Lspace^\oplus(\calA,\Qoppa)$ is $T(n)$-acyclic.
\end{prop}
\begin{proof}
  As before, we reduce to the case where we are generated by a single object. From Corollary \ref{cor: description_additive_L}
    \begin{align*}
    \Lspace^\oplus(\Proj^\omega\left(\Map_\calA(x,x)\right),\Qoppa_M) &\simeq \Lspace(\Stab(\Proj^\omega(\Map_\calA(x,x))),\Stab(\Qoppa)) \\
    &\simeq \Lspace(\Perf(\Map_\calA(x,x),\Stab(\Qoppa))\\
    &=\loopsinf\LL(\Map_\calA(x,x),\Stab(\Qoppa)).
\end{align*}
The final term is $T(n)$-acyclic by \cite[Corollary 15]{Lan22} for $n\geq 2$. The vanishing result for $n = 1$ is guaranteed by the assumption $p=2$ and Corollary 15 of \textit{loc. cit}. For $n=0$ we argue as follows: by Corollary \ref{cor: L_quadratic_Tn} we may reduce to the discrete case since $\Map_\calA(x,x)$ is connective. Observe that $\pi_0\Map_\calA(x,x)$ is $H\bbQ$-cyclic if and only if some multiple of $1 = [\id_x]$ is nullhomotopic. Suppose some $p'$-th multiple of $[\id_x]$ is nullhomotopic for some prime $p'$. By \cite[Proposition 4]{Lan22}
\[
\LL^q(\pi_0\Map_\calA(x,x)) \otimes \bbQ \to \LL^q(\pi_0\Map_\calA(x,x)[1/p']) \otimes \bbQ
\]
is an equivalence. But the right hand side is identically $0$.
\end{proof}
\begin{cor}\label{cor: mapping_spectra_stable_L_theory}
  Let $n \geq 0$. Suppose $(\calC,\Qoppa)$ is a stable Poincar\'e category. If for each object $x \in \calC$ the endomorphism spectrum $\map_\calA(x,x)$ is $\LTi$-acyclic,  then $\LL(\C,\Qoppa)$ is $T(n)$-acyclic.
\end{cor}
\begin{proof}
    If we assume $\calC$ is stable then using the formula 
    \[
   \Omega^{\infty-1}\LL (\calC,\Qoppa) \simeq |\Lspace^\oplus \Qdot(\calC,\Qoppa^{[1]})|
    \]
    of Corollary \ref{cor: description_additive_L}, the result will follow from Proposition \ref{prop: mapping_spaces_additive_L_theory} provided we can prove that the right hand side is $T(n)$-acyclic. This argument is as before.
\end{proof}

    In what follows, we consider  Poincar\'e categories of the form $(\Perf(A), \Qoppa)$ and investigate the effect that localisations of rings $\phi \from A \to LA$ have on the associated categories of (compact) modules. As discussed in \cite[Section 1.4]{CDHII}, such a localisation need not induce a Poincar\'e-Karoubi projection upon taking $\Perf(-)$. However we find   
    that if $\Qoppa$ is the quadratic structure associated to an involution $\sigma \from A \to A^\op$ and $\phi$ is a smashing localisation whose map on module categories has perfectly generated fibres then the induced map on Poincar\'e categories
    \[
    (\Perf(A), \Qoppa_\sigma^q) \to (\Perf(LA), \Qoppa_{L\sigma}^q) 
    \]
    is a Poincar\'e-Karoubi projection. These two hypotheses will typically be satisfied in the cases of interest to us.
        
\begin{lem}\label{lem: karoubi_sequence}
    For any $\Eone$-ring spectrum $A$ with anti-involution $\sigma \from A \to A^\op$, there is a Poincar\'e-Karoubi sequence
        \begin{equation}
             \begin{tikzcd}
            (\calC_{>n} \otimes \Perf(A),\Qoppa^q_\sigma|_{\calC_{>n} \otimes \Perf(A)}) \arrow[r] & (\Perf(A),\Qoppa^q_\sigma) \arrow[r] & (\Perf(\Lnf A),\Qoppa^q_{\Lnf \sigma}), 
        \end{tikzcd}\label{eqn: karoubi_sequence_module_cats}
        \end{equation} 
        and the endomorphism spectrum of every object in $\calC_{>n} \otimes \Perf(A)$ is $\Lnf$-acyclic. Here $\calC_{> n}$ denotes the category of $p$-local finite spectra which are $K(0) \oplus \cdots \oplus K(n)$-acyclic.
    \end{lem}
\begin{proof}
    On underlying stable categories, the Karoubi sequence property follows immediately from \cite[Lemma 3.7]{LMMT20}, essentially from the fact that $\Lnf$ is a smashing localisation and the fact that $\calC_{>n}$ is perfectly generated. In the same lemma in \textit{loc. cit.}, they show that every endomorphism spectrum of every object of $\calC_{>n} \otimes \Perf(A)$ is $\Lnf$-acyclic. It remains to investigate the Poincar\'e structures. But this is by design, the first is restricted from the second and the third is left Kan extended from the second by \cite[Corollary 1.4.6]{CDHII}.
\end{proof}

When we attempt to draw conclusions about $\LL$-theory from the previous lemma we quickly run into a problem. The sequence of Lemma \ref{lem: karoubi_sequence} is in general not Verdier and 
applying the $\LL$-theory functor to Equation \ref{eqn: karoubi_sequence_module_cats} does not necessarily yield a fibre sequence of spectra. Instead we may apply $\KaroubiL$, the Poincar\'e-Karoubi localising variant of $\LL$-theory to obtain the exact sequence
\begin{equation}
        \begin{tikzcd}
            \KaroubiL(\calC_{>n} \otimes \Perf(A),\Qoppa) \arrow[r] & \KaroubiL(\Perf(A),\Qoppa^q) \arrow[r] & \KaroubiL(\Perf(\Lnf A),\Qoppa^q). \label{eqn: exact_sequence_karoubi_L_theory}
        \end{tikzcd}
\end{equation} 

Given a Poincar\'e category $(\calC,\Qoppa)$, we will show that $ \LL (\calC,\Qoppa)$ and $ \KaroubiL(\calC,\Qoppa)$ can not be distinguished $T(n)$-locally for $n\geq 0$.

\begin{prop}[\cite{CDHIV}]\label{prop: cofib_L_Karoubi_L}
    There is a natural transformation of functors $\LL  \implies \KaroubiL\from \Catpidem \to \Sp$. Furthermore, for an idempotent complete Poincar\'e category $(\calC,\Qoppa)$ there is a cofibre sequence
    \begin{equation}
        \begin{tikzcd}
            \LL (\calC,\Qoppa) \arrow[r] &  \KaroubiL(\calC,\Qoppa) \arrow[r] & \left(\tau_{<0}\KaroubiK(\calC,\Qoppa)\right)^{\tCtwo}.
        \end{tikzcd}
    \end{equation}
\end{prop}

To investigate the  cofibre $\left(\tau_{<0}\KaroubiK(\calC,\Qoppa)\right)^{\tCtwo}$, we require a small lemma concerning the $T(n)$-acyclicity of the Tate construction.
\begin{lem}\label{lem: tate_bounded_above}
    For any bounded above spectrum $Y$ with a $C_2$-action, $Y^{\tCtwo} \otimes T(n) \simeq 0$ for $n \geq 1$. 
\end{lem}
\begin{proof} 
To see this we use the explicit formula for the Tate construction 
     \[
     Y^{\tCtwo} = \cofib\left( Y_{\hCtwo} \xrightarrow{\Nm_{C_2}}  Y^{\hCtwo} \right).
     \]
    Since the homotopy orbits are given by colimit, we have that $Y_{\hCtwo} \otimes  T(n) \simeq \left( Y \otimes  T(n)\right)_{\hCtwo}$ which vanishes. For the homotopy fixed points, we observe that if $Y$ is bounded above then so is $Y^{\hCtwo}$. Hence $Y^{\tCtwo} \otimes T(n)$ vanishes. 
\end{proof}

In fact, by \cite[Remark 22]{Lan22} the map $\LL(\calC,\Qoppa) \to \KaroubiL(\calC,\Qoppa)$ is also an equivalence after inverting $2$ so is in particular a rational equivalence.

\begin{cor}\label{cor: equiv_L_Karoubi_L}
    There is an equivalence $\LTn  \LL (\calC,\Qoppa) \simeq \LTn  \KaroubiL(\calC,\Qoppa)$ for $n \geq 0$.
\end{cor}
Next, suppose we have a diagram $F \from I \to \Catpd$ of categories with perfect duality. From \cite[Section B.1]{CDHII}, by postcomposing with the functor $(\calC, \mathbb D_\calC) \mapsto (\calC, \Qoppa^q_{\mathbb D_\calC})$ we obtain a diagram Poincar\'e categories. We call this diagram the ``diagram of quadratic Poincar\'e categories associated to $F$''.

\begin{lem}\label{lem: cartesian_square_categories_duality}
    Suppose 
    \[
\begin{tikzcd}
{(\mathcal{W}, \mathbb{D}_\mathcal{W})} \arrow[d, "F"] \arrow[r, "G"] & {(\mathcal{Y}, \mathbb{D}_\mathcal{Y})} \arrow[d, "F'"] \\
{(\mathcal{X}, \mathbb{D}_\mathcal{X})} \arrow[r, "G'"]               & {(\mathcal{Z}, \mathbb{D}_\mathcal{Z})}                
\end{tikzcd}
    \]
    is a Cartesian square of categories with perfect duality and assume that the underlying diagram of idempotent complete categories is a Karoubi square.  Then the associated diagram of quadratic Poincar\'e categories is Cartesian. 
\end{lem}
\begin{proof}
Since the underlying diagram of stable categories is Cartesian, it remains to prove that the induced Poincar\'e structure on the cone category is pulled back from the Poincar\'e structures on the rest of a diagram. For a category with perfect duality $(\calC, \mathbb{D}_\calC)$, the associated quadratic Poincar\'e structure is given by 
     $\Qoppa^q_{{\mathbb{D}}_\calC} \from x \mapsto \map_\calC(c,\mathbb{D}_\calC(c))_{\hCtwo}$. By assumption, for a given $w \in \mathcal{W}$, we have a $\Ctwo$-equivariant diagram 
\[
\begin{tikzcd}
{\map_\mathcal{W}(w , \mathbb{D}_\mathcal{W}(w))} \arrow[d] \arrow[r]    & {\map_\mathcal{Y}(G(w) , (G \circ \mathbb{D}_\mathcal{W})(w))} \arrow[d]            \\
{\map_\mathcal{X}(F(w) , (F \circ \mathbb{D}_\mathcal{W})(w))} \arrow[r] & \begin{tabular}{c}
    $ {\map_\mathcal{Z}((F' \circ G)(w) , (F' \circ G \circ \mathbb{D}_\mathcal{W})(w))} $  \\
    $\simeq$ \\
    $ {\map_\mathcal{Z}((G' \circ F)(w) , (G ' \circ F \circ \mathbb{D}_\mathcal{W})(w))}  $
\end{tabular}
\end{tikzcd}
\]
which is Cartesian on underlying spectra. This diagram is equivalent to
\[
\begin{tikzcd}
{\map_\mathcal{W}(w , \mathbb{D}_\mathcal{W}(w))} \arrow[d] \arrow[r]    & {\map_\mathcal{Y}(G(w) , ( \mathbb{D}_\mathcal{Y} \circ G^\op)(w))} \arrow[d]            \\
{\map_\mathcal{X}(F(w) , (\mathbb{D}_\mathcal{X} \circ F^\op)(w))} \arrow[r] & \begin{tabular}{c}
    $ {\map_\mathcal{Z}((F' \circ G)(w) , (\mathbb{D}_\mathcal{Z} \circ F'^\op \circ G^\op )(w))} $  \\
    $\simeq$ \\
    $ {\map_\mathcal{Z}((G' \circ F)(w) , (\mathbb{D}_\mathcal{Z} \circ G'^\op \circ F^\op )(w))}  $
\end{tabular}
\end{tikzcd}.
\]
Applying the exact functor $(-)_{\hCtwo}$, we obtain the desired result.
\end{proof}

We will apply the above lemma to obtain an analogue of the chromatic fracture square.

\begin{cor}\label{cor: cocartesian_square_rings_involution}
  Suppose $(A,\sigma)$ is an $\Eone$-ring with anti-involution. Suppose $n \geq 1$ and let $T = T(1) \oplus \cdots \oplus T(n)$. Then the following is a Poincar\'e-Karoubi square:
    \[ 
    \begin{tikzcd}
{(\Perf(\Lnf A) , \Qoppa_{\Lnf \sigma}^q)} \arrow[rr] \arrow[d] &  & {(\Perf (L_T A) , \Qoppa_{L_T \sigma}^q)} \arrow[d] \\
{(\Perf (A[1/p]) , \Qoppa_{\sigma [1/p]}^q)} \arrow[rr]                               &  & {(\Perf{L_T A} [1/p] , \Qoppa_{L_T \sigma [1/p]}^q)}
\end{tikzcd} .
    \]
\end{cor}
\begin{proof}
   Since the diagram in the statement arises from a Cartesian square of categories with perfect duality, by Lemma \ref{cor: cocartesian_square_rings_involution} it gives rise to the Cartesian square in $\Catpidem$, since limits in $\Catpidem$ are computed as in $\Catp$. Finally we observe that the vertical arrows are Poincar\'e-Karoubi projections by analogy with Lemma \ref{lem: karoubi_sequence}.
\end{proof}
\subsection{Proofs of Theorems \ref{thm: A} and \ref{thm: B}}
With the necessarily preliminaries complete, we are ready to prove Theorems \ref{thm: A}, \ref{thm: B}, and \ref{thm: C}.

\begin{proof}[Proof of Theorem \ref{thm: purity}]
\hfill
\begin{enumerate}
    \item We rely on three previous results. The the Poincar\'e-Karoubi sequence of Lemma \ref{lem: karoubi_sequence} induces an exact sequence of spectra 
    \[
    \KaroubiL(\calC_{> n} \otimes \Perf(A), \Qoppa_\sigma^q) \to \KaroubiL^q_\sigma(\Perf(A)) \to \KaroubiL^q_{\Lnf \sigma}(\Perf(\Lnpf A)).
    \]
    By Corollary \ref{cor: equiv_L_Karoubi_L}, after $T(n)$-localisation for $n \geq 0$, we may replace $\LTn\KaroubiL(-)$ with $\LTn\LL(-)$ and so we obtain an exact sequence
    \[
    \LTn\LL^q_\sigma(\calC \otimes \Perf(A)) \to \LTn\LL^q_\sigma(\Perf(A)) \to \LTn\LL^q_{\Lnf \sigma}(\Perf(\Lnpf A)).
    \]
The first term above vanishes by Proposition \ref{cor: mapping_spectra_stable_L_theory}.
    \item  We apply $\LTn \LL(-)$ for $n\geq1$ to the Poincar\'e-Karoubi square of Corollary \ref{cor: cocartesian_square_rings_involution} obtain the Cartesian square 
        \[ 
\begin{tikzcd}
\LTn\LL^q(\Lnf A) \arrow[rr] \arrow[d]            &  & \LTn\LL^q(\LL_TA) \arrow[d]        \\
{\LTn\LL^q(A[1/p])} \arrow[rr] &  & {\LTn\LL^q({L_T A} [1/p])}
\end{tikzcd}.
    \]
    Here we suppress the Poincar\'e structures for readability. Both spectra on the bottom row are $0$ - they are modules over $\LTn \LL(\bbS[1/p],\Qoppa^s_\id)$. However, $\bbS[1/p]$ is $T(i)$-acyclic for all $i \geq 1$ and so its $T(n)$-local $\LL$-theory vanishes by \cite{Lan22}. Hence the top horizontal arrow is an equivalence.
\end{enumerate}
\end{proof}
From this result, we deduce Theorem \ref{thm: B}.
\begin{proof}[Proof of Theorem \ref{thm: B}]
    The case $n=0$ is immediate from Part 1 of Theorem \ref{thm: A}.  We prove the result for $n\geq 1$. Consider the diagram of $\Eone$-rings with anti-involution,
    \[
\begin{tikzcd}
{(\tau_{\geq 0}A, \tau_{\geq 0} \sigma)} \arrow[d, ] \arrow[r] & {(A,  \sigma)} \arrow[d, ] \\
{(L_T \tau_{\geq 0}A, L_T \tau_{\geq 0} \sigma)} \arrow[r, ]       & {(L_TA,  L_T\sigma)}            
\end{tikzcd}
    \]
     Our previous results demonstrate that all maps above induce equivalences on $T(n)$-local quadratic $\LL$-theory. For the two vertical maps this follows from Theorem \ref{thm: A}. For this bottom horizontal map this is due to the natural equivalence of functors $L_T \tau_{\geq 0} \simeq L_T$. Thus the same is true of the top vertical map $(\tau_{\geq 0}A, \tau_{\geq 0} \sigma) \to (A,  \sigma)$. From \cite[Corollary 15]{Lan22} we know that $(\tau_{\geq 0} A, \tau_{\geq 0} \sigma)$ has vanishing $T(n)$-local quadratic $\LL$-theory. Thus all terms above have vanishing $T(n)$-local quadratic $\LL$-theory. By Theorem \ref{thm: magic_formula}, the same is true for Poincar\'e structure on $\Perf(A)$ compatible with $\sigma$. 
    \end{proof}

\newpage

\appendix

\section{Functoriality of Additive Poincar\'e categories}\label{appendix: functoriality}

We prove a variety of results allowing us to better understand limits and colimits in $\flat$-additive Hermitian and Poincar\'e categories. Many of the proofs we supply are essentially the same as those that appear in \cite[Section 1, Section 6]{CDHI}. As such, we shall refer extensively to the corresponding statements therein, giving the necessarily adjustments to the arguments to ensure their proofs go through in our setup.  To this end, we begin by establishing some functoriality results.

\begin{lem}\label{lem: LKE_additive}
 Let $f \from \calA \to \calB$ and $g \from \calC \to \calD$ be additive functors between $\flat$-additive categories. Then
\begin{enumerate}
    \item the left Kan extension functor
    \[
    f_! \from \Fun(\calA^\op, \CGrp) \to \Fun(\calB^\op, \CGrp)
    \]
    sends additive functors to additive functors;
    \item the left Kan extension functor
    \[
    (f \times g)_! \from \Fun(\calC^\op \times \calA^\op, \CGrp) \to \Fun(\calD^\op \times \calB^\op, \CGrp)
    \]
    sends additive bilinear functors to additive bilinear functors.
    \item the left Kan extension functor
    \[
    f_! \from \Fun(\calA^\op , \CGrp) \to \Fun(\calB^\op , \CGrp)
    \]
    sends additive quadratic functors to additive quadratic functors.
\end{enumerate}
\end{lem}

\begin{proof}
We follow the proof of \cite[Lemma 1.4.1]{CDHI} closely. 
\begin{enumerate}
    \item[(1) + (3)]We embed $\calA$ and $\calB$ into their respective free cosifted limit completions. This amounts to the diagram
\[
\begin{tikzcd}
\calA \arrow[d, "i"] \arrow[r, "f"] & \calB \arrow[d, "j"] \\
\left(\calP_\Sigma(\calA^\op)\right)^\op \arrow[r, "\tilde f"]             & \left(\calP_\Sigma(\calB^\op)\right)^\op          
\end{tikzcd}
\]
Let $g$ be left adjoint to $\tilde f$. Note that $g$ is additive and corresponds to restriction along $f$. Suppose $F \from \calC^\op \to \CGrp$ is an arbitrary functor. Upon taking $\Fun(-^\op,\CGrp)$ we get 
\[
f^\op_! F \simeq (j^\op)^* j^\op_! f^\op_!F \simeq (j^\op)^* \tilde f_!^\op i_!^\op F.
\]
Furthermore, $\tilde f^\op_! \simeq (g^\op)^*$ by uniqueness of adjoints. Plugging this into the earlier formula gives 
\[
f^\op_! F \simeq (j^\op)^* \tilde f_!^\op i_!^\op F \simeq (j^\op)^* (g^\op)^* i_!^\op F.
\]
Observe that $i_!^\op$ is inverse to the equivalence that witnesses $\calP_\Sigma(\calC^\op)$ as the free cosifted limit completion of $\calC$. In particular, $i_!^\op$ sends product preserving functors to product preserving functors. After precomposing by the finite coproduct preserving functors $g^\op$ and $j^\op$, we conclude that $(j^\op)^* (g^\op)^* i_!^\op F$ is additive. More generally, we can use the diagram, the equivalences coming from \cite[Theorem 2.19]{BGMN22},
\[
 \begin{tikzcd}
{\Fun_{\leq n}(\calA,\Sp)}                  & {\Fun_{\leq n}(\Stab(\calA),\Sp)} \arrow[l, "\simeq"]                         \\
{\Fun_{\leq n}(\calB,\Sp)} \arrow[u, "f^*"] & {\Fun_{\leq n}(\Stab(\calB),\Sp)} \arrow[l, "\simeq"] \arrow[u, "\Stab(f)^*"]
\end{tikzcd}.
\]
Passing to vertical left adjoints and using \cite[Proposition 6.1.5.4]{HA} we conclude that $i_!^\op F$ is additive (respectively quadratic) if $F$ is.
    \item[2] By symmetry the identity $(f^\op \times g^\op)_! \simeq (f^\op \times \id)_!(\id \times g^\op)$ we may assume that $f = \id$ and consider only $g$. We have that $(\id \times g)_!F|_{\{a\} \times \calB^\op} \simeq g_!(F|_{\{a\}\times \calA}$. Using the equivalences
    \begin{align*}
        \Fun(\calC^\op \times \calA^\op, \CGrp)  \simeq \Fun(\calC^\op,\Fun(\calA^\op,\CGrp), \\
        \Fun(\calC^\op \times \calB^\op, \CGrp)  \simeq \Fun(\calC^\op,\Fun(\calB^\op,\CGrp),
    \end{align*}
    we identify additive bilinear functors with additive functors to the (additive) category of additive functors. Furthermore, under this equivalence, $(\id \times g)_!$ identifies with post-composition with $g_!$. By the previous part of the lemma, $g_!$ this is additive and sends additive functors to additive functors.

\end{enumerate} 
\end{proof}

We obtain the following two corollaries.
\begin{cor}\label{cor: catah_catadd)cocartesian}
    The projection
    \[
    \Catah \to \Catadd
    \]
    is a cocartesian fibration, with pushforward along $f \from \calA \to \calB$ given by $\Qoppa \mapsto f_! \Qoppa$.
\end{cor}

According to the general principle guaranteeing (co)limits in total spaces of (co)Cartesian fibrations, we may argue as in \cite[Proposition 6.1.2]{CDHI}, to obtain the existence of limits and colimits in $\Catah$.
\begin{cor}\label{cor: catah_limits_colimits}
    The $\infty$-category $\Catah$ admits all limits and colimits and these are preserved by the forgetful functor $\pi \from \Catah \to \Catadd$.
\end{cor}
In our situation, where a (certain subcategory of a) functor category is bifibred over a base category with limits and colimits, one has an explicit recipe for computing these limits and colimits as given in \cite[Remark 6.1.3]{CDHI}.  The next is the analogue of \cite[Proposition 1.4.3]{CDHI}. Fortunately the arguments given in \textit{loc. cit.} continue to hold.
\begin{lem}
    Let $f \from \calA \to \calB$ be an additive functor between $\flat$-additive categories and let $\Qoppa$ be an additive quadratic functor on $\calA$. Then there is a canonical natural equivalence
\[
(f \times f)_!\rmB_\Qoppa \implies \rmB_{f_! \Qoppa} 
\]
\end{lem}
\begin{proof}
     The argument used to prove \cite[Proposition 1.4.3]{CDHI} goes through, with some modification. We can reduce to the case $\calB = (\PP_\Sigma(\calA^\op))^\op$ and cosifted systems $(x_i)_{i \in I}$ and $(y_j)_{j \in J}$ Following their proof we must show that 
    \begin{align*}
         \colim_{(i,j) \in I^\op \times \J^\op} \fib\left[ \Qoppa(x_i \oplus y_j) \to \Qoppa(x_i) \oplus \Qoppa(y_j)\right] \\\
        \to  \fib\left[  \colim_{(i,j) \in I^\op \times \J^\op}  \Qoppa(x_i \oplus y_j) \to  \colim_{(i\in I^\op} \Qoppa(x_i) \oplus  \colim_{j \in  \J^\op} \Qoppa(y_j) \right]
    \end{align*}
   is an equivalence. We claim that this is equivalent to the following 
   \begin{align*}
         \colim_{(i,j) \in I^\op \times \J^\op} \fib\left[ \Qoppa(x_i \oplus y_j) \to \Qoppa(x_i) \oplus \Qoppa(y_j)\right] \\\
        \to  \fib\left[  \colim_{(i,j) \in I^\op \times \J^\op}  \Qoppa(x_i \oplus y_j) \to  \colim_{(i,j) \in I^\op \times \J^\op} \Qoppa(x_i) \oplus  \colim_{(i,j) \in I^\op \times \J^\op} \Qoppa(y_j) \right]
    \end{align*}
    is an equivalence. To see this, observe that $I^\op \times J^\op \to I^\op$ (similarly for the projection to $J^\op$) is cofinal. By Quillen's Theorem $A$, we must show that $(I^\op \times J^\op) \times_{I^\op} \times I^\op_{i/}$ is weakly contractible. This simplicial set is equivalent to $J^\op \times I^\op_{i/}$. Taking geometric realisations we obtain 
    \[
    |J^\op \times I^\op_{i/}| \simeq  |J^\op| \times |I^\op_{i/}| \simeq * \times |I_{i/}| \simeq * \times * \simeq *
    \]
     from \cite[Proposition 5.5.8.7]{HTT}. The theorem will now follow if we can exchange certain limits and colimits in $\CGrp$, which we identify with $\Spcn$. Because $S^0$ is compact projective in $\Spcn$, we observe that we have surjections on $\pi_0$. Therefore the fibre taken in $\Spcn$ coincides with the fibre taken in $\Sp$ and the limit-colimit exchange is justified. 
\end{proof}

In this section we record three results, the additive analogues of \cite[Proposition 6.1.4, 6.1.5, 6.1.7]{CDHI}. The proofs are rather long but go through with little modification and so are omitted. The reader may verify that the proofs only depend on the results just we have proven and are otherwise formal.
\begin{prop}[Closure properties of $\Catap$ in $\Catah$]\label{prop: closure_properties_catap_catah}
    \hfill
    \begin{enumerate}
        \item Let $\overline p \from K^\rhd  \to \Catadd$ be a colimit diagram of $\flat$-additive categories and let $f \from \A_\infty \defeq  \overline p (\infty) \to \B$ be an additive functor to a cocomplete $\flat$-additive category $\B$. Then the canonical map
        \[
        \colim_{k \in K}(i_k)_! i_k^* f \to f
        \]
        is an equivalence.
        \item The category $\Catap$ has all small limits and colimits and the inclusion $\Catap \to \Catah$ preserves all small limits and colimits.
        \item The categories $\Catadd, \Catah, \Catap$ are all semiadditive.
    \end{enumerate}
\end{prop}

We record an easy corollary of the previous results that we will have need for later.
\begin{lem}\label{lem: addititve_poincare_infty_cat_generated_finite_subcats} 
Suppose $(\calA,\Qoppa)$ is a $\flat$-additive category. Then  $(\calA,\Qoppa)$ is a filtered colimit over its $\flat$-additive subcategories generated by finitely many elements under direct sums and retracts and which are closed under the duality $\DQoppa$. More precisely, if let $K$ be the poset of full $\flat$-additive subcategories $\calA_k$ of $\calA$ generated by finitely many elements and closed under duality of $\calA$.
\[
(\calA,\Qoppa) \simeq \colim_{K} (\calA_k, \Qoppa|_{\Qoppa_{\calA_k}})
\]
\end{lem}
The next result is the adaptation of standard Morita theory-type arguments to classify symmetric biadditive functors on $\Proj^\omega(R)$.
\begin{prop}\label{prop: additive_hermitian_schwede-shipley}
 Let $(\A,\Qoppa)$ be a small idempotent complete additive Hermitian category whose underlying additive category $\A$ is generated by a single element. Then $(\A,\Qoppa) \simeq (\Proj^\omega(R), \Qoppa_M)$ for $R$ a connective $\Eone$-ring with associated symmetric biadditive functor $\rmB_{\Qoppa_M}(X,Y) \simeq \Map_{R \otimes R}(X \otimes Y, M)$ for an $(R \otimes R)$-module $M$ equipped with a $C_2$-action compatible with the flip action on $R \otimes R$. 
\end{prop}
\begin{proof}
Since $\A$ is generated by a single element we conclude from the ordinary additive Schwede-Shipley theorem that $\calA \simeq \Proj^\omega(R)$.  Next we consider biadditive functors from such categories. First we remark that we may identify $\Proj^\omega(R)^\op$ with $\Proj^\omega(R^\op)$ via $M \mapsto \Map_R(M,R)$. Consider the equivalences
\begin{align*}
    \Fun^{\mathrm{biadd}}(\Proj^\omega(R^\op )\times \Proj^\omega(R^\op) ,\CGrp) &\simeq \Fun^\amalg(\Proj^\omega(R^\op) \otimes^\amalg \Proj^\omega(R^\op) , \CGrp)  \\
    &\simeq \Fun^\amalg(\Proj^\omega(R^\op \otimes R^\op) , \CGrp) .
\end{align*}
We know that since $\Proj^\omega(R^\op \otimes R^\op)$ is semiadditive, we have the equivalence
\[
\Fun^\amalg(\Proj^\omega(R^\op \otimes R^\op) , \CGrp) \simeq \Fun^\Pi(\Proj^\omega(R^\op \otimes R^\op) , \CGrp).
\]
Using semiadditivity once again, we have that 
\[
 \Fun^\Pi(\Proj^\omega(R^\op \otimes R^\op) , \CGrp) \simeq \Fun^\Pi(\Proj^\omega(R^\op \otimes R^\op) , \Spaces)
\]
and the latter is precisely the definition of $\calP_\Sigma((\Proj^\omega(R^\op \otimes R^\op)^\op)$. Identifying $\Proj^\omega(R^\op \otimes R^\op)$ with $\Proj^\omega(R \otimes R)$ we obtain that 
\[
 \Fun^{\mathrm{biadd}}(\Proj^\omega(R^\op )\times \Proj^\omega(R^\op)) \simeq \calP_\Sigma(\Proj^\omega(R \otimes R)) \simeq \Mod(R \otimes R)^\mathrm{cn}
\]
where $\Mod(R \otimes R)^\mathrm{cn}$ is the smallest full subcategory of $\Mod(R \otimes R)$ containing $R \otimes R$ and closed under arbitrary colimits and retracts. This equivalence follows from  \cite[Example C.1.5.11]{SAG}. Finally, as with the stable case, the condition for a biadditive functor $F$ to be symmetric translates to the condition that the corresponding $(R\otimes R)$-module $M$ is endowed with the structure of a $C_2$-action compatible with the flip action on $R \otimes R$.
\end{proof}

\printbibliography
\end{document}